\documentclass[11pt,english]{amsart}

\usepackage{amsthm}
\usepackage{amstext}
\usepackage{amssymb}

\theoremstyle{plain}
\newtheorem{thm}{Theorem}
  \theoremstyle{plain}
  \newtheorem{conjecture}[thm]{Conjecture}
  \theoremstyle{definition}
  \newtheorem{defn}[thm]{Definition}
  \theoremstyle{plain}
  \newtheorem{lem}[thm]{Lemma}
  \theoremstyle{plain}
  \newtheorem*{thm*}{Theorem}

\makeatother

\usepackage{babel}

\def\COMMENT#1{}
\def\eps{\varepsilon}
\let\epsilon=\varepsilon
\let\subset=\subseteq

\begin{document}

\def\noproof{{\unskip\nobreak\hfill\penalty50\hskip2em\hbox{}\nobreak\hfill%
        $\square$\parfillskip=0pt\finalhyphendemerits=0\par}\goodbreak}
\def\endproof{\noproof\bigskip}
\newdimen\margin   
\def\textno#1&#2\par{%
    \margin=\hsize
    \advance\margin by -4\parindent
           \setbox1=\hbox{\sl#1}%
    \ifdim\wd1 < \margin
       $$\box1\eqno#2$$%
    \else
       \bigbreak
       \hbox to \hsize{\indent$\vcenter{\advance\hsize by -3\parindent
       \sl\noindent#1}\hfil#2$}%
       \bigbreak
    \fi}
\def\proof{\removelastskip\penalty55\medskip\noindent{\bf Proof. }}



\global\long\def\labelenumi{(\roman{enumi})}

\title{Optimal packings of Hamilton cycles in graphs of high minimum degree}

\author{Daniela K\"uhn, John Lapinskas and Deryk Osthus}

\thanks{D.~K\"uhn was supported by the ERC, grant no.~258345.}

\date{\today}
\begin{abstract}
We study the number of edge-disjoint Hamilton cycles one can guarantee in
a sufficiently large graph $G$ on $n$ vertices with minimum degree $\delta=(1/2+\alpha)n$.
For any constant $\alpha>0$, we give an optimal answer in the following sense: 
let $\textnormal{reg}_{\textnormal{even}}(n,\delta)$ denote the degree of the 
largest even-regular spanning subgraph one can guarantee in a graph on $n$ vertices
with minimum degree $\delta$. Then the number of edge-disjoint Hamilton cycles we
find equals $\textnormal{reg}_{\textnormal{even}}(n,\delta)/2$.
The value of  $\textnormal{reg}_{\textnormal{even}}(n,\delta)$ is known for infinitely many 
values of $n$ and $\delta$.
We also extend our results to graphs $G$ of minimum degree $\delta \ge n/2$,
unless $G$ is close to the extremal constructions for Dirac's theorem.
Our proof relies on a recent and very general result of K\"uhn and Osthus on Hamilton
decomposition of robustly expanding regular graphs.
\end{abstract}
\maketitle

\section{Introduction}\label{sec:intro}

Dirac's theorem \cite{Dirac} states that any graph on $n\ge3$ vertices
with minimum degree at least $n/2$ contains a Hamilton cycle. This
degree condition is best possible. Surprisingly, though, the assertion of Dirac's theorem
can be strengthened considerably: Nash-Williams~\cite{Diracext}
proved that the conditions of Dirac's theorem actually 
guarantee linearly many edge-disjoint Hamilton cycles.
\begin{thm}
Every graph on $n$ vertices with minimum degree at least $n/2$ contains
at least $\lfloor5n/224\rfloor$ edge-disjoint Hamilton cycles.
\end{thm}
Nash-Williams \cite{initconj} initially conjectured that such a graph
must contain at least $\lfloor n/4\rfloor$ edge-disjoint Hamilton
cycles, which would clearly be best possible. However, Babai observed that
this trivial bound is very far from the truth (see \cite{initconj}).  
Indeed, the following construction (which is based on Babai's argument) 
gives a graph $G$ which contains at most $\lfloor(n+2)/8\rfloor$
edge-disjoint Hamilton cycles. The graph $G$ consists of one empty
vertex class $A$ of size $2m$, one vertex class $B$ of size $2m+2$ containing
a perfect matching and no other edges, and all possible edges between $A$ and $B$.
Thus $G$ has order $n=4m+2$ and minimum degree $2m+1$. 
Any Hamilton cycle in $G$ must contain at least two edges of
the perfect matching in~$B$, so $G$ contains at most $\lfloor(m+1)/2\rfloor$ edge-disjoint
Hamilton cycles. 

The above question of Nash-Williams naturally extends to graphs of higher minimum degree:
suppose that $n/2 \le \delta \le n-1$.
\emph{How many edge-disjoint Hamilton cycles can one guarantee in a graph $G$ on $n$ vertices with minimum degree $\delta$?}

Clearly, as $\delta$ increases, one expects to find more edge-disjoint Hamilton cycles.
However, the above construction shows that the trivial bound of  $\lfloor \delta/2 \rfloor$ cannot always be attained. 
A less trivial bound is provided by the largest even-regular spanning subgraph in~$G$. More precisely,
let $\textnormal{reg}_{\textnormal{even}}(G)$
be the largest degree of an even-regular spanning subgraph of $G$. Then let
\[
\textnormal{reg}_{\textnormal{even}}(n,\delta):=\min\{\textnormal{reg}_{\textnormal{even}}(G):|G|=n,\ \delta(G)=\delta\}.
\]
Clearly, in general we cannot guarantee more than $\textnormal{reg}_{\textnormal{even}}(n,\delta)/2$
edge-disjoint Hamilton cycles in a graph of order $n$ and minimum
degree $\delta$. In fact, we conjecture this bound can always be
attained.
\begin{conjecture} \label{con:mainconjecture}
Suppose $G$ is a graph on $n$ vertices
with minimum degree $\delta \ge n/2$. Then $G$ contains at least
$\textnormal{reg}_{\textnormal{even}}(n,\delta)/2$ edge-disjoint
Hamilton cycles.
\end{conjecture}
Our main result confirms this conjecture exactly, as long as  $n$ is large and
$\delta$ is slightly larger than $n/2$.
\begin{thm}\label{thm:mainresult}
For every $\eps>0$, there exists an
integer $n_{0}=n_{0}(\epsilon)$ such that every graph $G$ on $n\ge n_{0}$
vertices with $\delta(G)\ge(1/2+\eps)n$ contains at least
$\textnormal{reg}_{\textnormal{even}}(n,\delta(G))/2$ edge-disjoint
Hamilton cycles.
\end{thm}
In fact, we even show that if $G$ is not close to the extremal example, then $G$ contains 
significantly more than the required number of edge-disjoint Hamilton cycles
(see Lemma~\ref{lem:mainresultpart2}).
Our proof of Theorem~\ref{thm:mainresult} is based on a recent result (Theorem~\ref{thm:hamdecresult}) 
of K\"uhn and Osthus~\cite{Kelly,Kellyapps},
which states that every ``robustly expanding'' regular (di)graph has a Hamilton decomposition.
In~\cite{Kellyapps}, a straightforward argument was already used to derive Conjecture~\ref{con:mainconjecture}
for $\delta\ge(2-\sqrt{2}+\epsilon)n$ (see Section~\ref{toolsproof}). 
Our extension of this result to $\delta \ge(1/2+\eps)n$ involves new ideas.

Subsequently, Csaba, K\"uhn, Lo, Osthus and Treglown~\cite{1factor} have proved Conjecture~\ref{con:mainconjecture} for large $n$,
by solving the case when $\delta$ is allowed to be close to $n/2$.
The proof relies on Theorem~\ref{thm:mainresult} and Theorem~\ref{thm:halfnresult}. (The latter provides a stability result
when $\delta$ is close to $n/2$.)

Earlier, Christofides, K\"uhn and Osthus~\cite{CKO} used the regularity lemma
to prove an approximate version of Theorem~\ref{thm:mainresult}.
Hartke and Seacrest~\cite{HartkeHCs} were able improve this result while avoiding the  use of the regularity
lemma (but still with the same restriction on $\delta$). 
This enabled them to omit the condition that $G$ has to be very large. They also gave significantly better error bounds.

Accurate bounds on $\textnormal{reg}_{\textnormal{even}}(n,\delta)$ are known. 
Note that the complete bipartite graph whose vertex classes are almost equal shows that 
$\textnormal{reg}_{\textnormal{even}}(n,\delta)=0$ for $\delta<n/2$.
Katerinis \cite{Katerinis} considered the case when
$\delta=n/2$. His result  was independently generalised to larger values of $\delta$ in~\cite{CKO} 
(see~\cite{Kellyapps} for a summarised
version) and by Hartke, Martin and Seacrest \cite{Hartkefactors}.
The following bounds are from~\cite{Hartkefactors}.
\begin{thm}\label{thm:regbounds}
Suppose that $n,\delta\in\mathbb{N}$ and $n/2\le\delta<n$.
Then
\begin{equation} \label{regbound}
\frac{\delta+\sqrt{n(2\delta-n)+8}}{2}-\epsilon\le\textnormal{reg}_{\textnormal{even}}(n,\delta)
\le\frac{\delta+\sqrt{n(2\delta-n)}}{2}+\frac{4}{\sqrt{n(2\delta-n)}+4}.
\end{equation}
where $0<\epsilon \le 2$ is chosen to make the left hand side of~(\ref{regbound}) an even
integer.
\end{thm}
Note that~(\ref{regbound}) always yields at most two possible values 
for $\textnormal{reg}_{\textnormal{even}}(n,\delta)$ and even determines it
exactly for many values of the parameters $n$ and $\delta$. For example,~(\ref{regbound}) determines
$\textnormal{reg}_{\textnormal{even}}(n,n/2)$ (e.g.~in the case when $n$ is divisible by 8 it is $n/4$).
The bounds in~\cite{CKO} also give at most two possible values.
The lower bound in~(\ref{regbound}) is based on Tutte's factor theorem~\cite{Tutte}.
The upper bound is obtained by a natural generalization of Babai's construction
(see Section~\ref{sec:extexample} for a description).

Our second result concerns the case of Conjecture~\ref{con:mainconjecture} where we allow $\delta$ to be close to $n/2$.
In this case, we obtain the following `stability result':
if $\delta(G)=(1/2+o(1))n$, then
Conjecture~\ref{con:mainconjecture} holds for large $n$ as long as $G$ has
suitable expansion properties. 
In this case, we even obtain significantly more than the required number of edge-disjoint Hamilton cycles again.
These expansion properties fail only when $G$
is very close to the extremal examples for Dirac's theorem.
\begin{thm}\label{thm:halfnresult}
For every $0<\eta<1/8$, there exist $\eps>0$ and an integer $n_{0}$ such that
every graph $G$ on $n\ge n_{0}$ vertices with $(1/2-\eps)n\le \delta(G)\le(1/2+\eps)n$
satisfies one of the following:
\begin{enumerate}
\item There exists $A\subset V(G)$ with $|A|=\lfloor n/2\rfloor$ and such that
$e(A)\le\eta n^{2}$.
\item There exists $A\subset V(G)$ with $|A|=\lfloor n/2\rfloor$ and such that
$e(A,\overline{A})\le\eta n^{2}$.
\item $G$ contains at least $\max \{\textnormal{reg}_{\textnormal{even}}(n,\delta(G))/2,n/8 \}+\eps n$
edge-disjoint Hamilton cycles.
\end{enumerate}
\end{thm}
Note that if $G$ satisfies (i) then $e(A,\overline{A})$ must be roughly $n^2/4$,
i.e.~$G$ is close to $K_{n/2,n/2}$ with possibly some edges added to one of the vertex classes.
If $G$ satisfies~(ii), then both $e(A)$ and $e(\overline{A})$ must be roughly $n^2/8$, i.e.~$G$ is close to the union of two equal-sized cliques.

Although Conjecture~\ref{con:mainconjecture} is optimal for the class of graphs on $n$ vertices and minimum degree $\delta$,
it will not be optimal for every graph in the class -- some graphs $G$ will contain far more than 
$\textnormal{reg}_{\textnormal{even}}(n,\delta)/2$ edge-disjoint
Hamilton cycles. The following conjecture accounts for this and would be best possible for every single graph~$G$.
Note that it is far stronger than Conjecture~\ref{con:mainconjecture}.
\begin{conjecture} \label{con:betterconj}
Suppose $G$ is a graph on $n$ vertices
with minimum degree $\delta(G) \ge n/2$. Then $G$ contains at least
$\textnormal{reg}_{\textnormal{even}}(G)/2$ edge-disjoint
Hamilton cycles.
\end{conjecture}
For $\delta\ge(2-\sqrt{2}+\epsilon)n$, this conjecture was proved in~\cite{Kellyapps}, based on the main result of~\cite{Kelly}.
It would already be very interesting to obtain an approximate version of Conjecture~\ref{con:betterconj},
i.e.~a set of $(1-\eps)\textnormal{reg}_{\textnormal{even}}(G)/2$ edge-disjoint Hamilton cycles under the assumption that
$\delta(G) \ge (1+\eps)n/2$.

As a very special case, Conjecture~\ref{con:betterconj} would imply the long-standing `Hamilton factorization' conjecture of 
Nash-Williams \cite{initconj,decompconj}:
any $d$-regular graph on at most $2d$ vertices contains $\lfloor d/2\rfloor$
edge-disjoint Hamilton cycles.%
\COMMENT{I don't see off the top of my head why $2d$ vertices is best possible
here. Certainly if you have $2d+2$ then you can make the graph disconnected,
but with $2d+1$ that doesn't work. Also, the canonical example of
a graph on $n$ vertices with minimum degree $\frac{n-1}{2}$ but
no Hamilton cycle is $K_{\frac{n+1}{2},\frac{n-1}{2}}$ which isn't
regular.}
Jackson~\cite{decompconj} raised the same conjecture independently,
and proved a partial result. This was improved to an approximate version of the conjecture in~\cite{CKO}. 
The best current result towards the Hamilton factorization conjecture is due to
K\"uhn and Osthus~\cite{Kellyapps} (again as a corollary of their main result in~\cite{Kelly}).
Note that the set of Hamilton cycles guaranteed by Theorem~\ref{thm:reg}
actually forms a Hamilton decomposition.
\begin{thm}\label{thm:reg}
For every $\eps>0$ there exists an integer $n_{0}$ such that every
$d$-regular graph on $n\ge n_{0}$ vertices for which $d\ge(1/2+\eps)n$ is even contains
$d/2$ edge-disjoint Hamilton cycles.
\end{thm}
Frieze and Krivelevich conjectured that the trivial bound of $\lfloor \delta(G)/2 \rfloor$ edge-disjoint 
Hamilton cycles is in fact correct for random graphs.
Indeed, the results of several authors (mainly Krivelevich and Samotij~\cite{KrS}
as well as Knox, K\"uhn and Osthus~\cite{Knox2011e}) can be combined to show that  for all $0\le p\le 1$, 
the binomial random graph $G_{n,p}$ contains
$\lfloor\delta(G_{n,p})/2\rfloor$ edge-disjoint Hamilton cycles with
high probability. 
Some further related results can be found in~\cite{Gnpsurvey,Kelly,Kellyapps}.

\section{Notation}\label{sec:notation}

Given a graph $G$, we write $V(G)$ for its vertex set, $E(G)$ for
its edge set, $e(G):=|E(G)|$ for the number of its edges and $|G|$
for the number of its vertices. Given $X\subset V(G)$, we write $G-X$
for the graph formed by deleting all vertices in $X$ and $G[X]$
for the subgraph of $G$ induced by $X$. We will also write $\overline{X}:=V(G)\setminus X$
when it is unambiguous to do so. Given disjoint sets $X,Y\subset V(G)$,
we write $G[X,Y]$ for the bipartite subgraph induced by $X$ and
$Y$. If $G$ and $G'$ are two graphs, we write $G\dot{\cup}G'$ for the
graph on $V(G)\dot{\cup}V(G')$ with edge set $E(G)\dot{\cup}E(G')$.
If $V(G)=V(G')$, we also write $G+G'$ for the graph on $V(G)$ with
edge set $E(G)\cup E(G')$. An \emph{$r$-factor} of a graph $G$ is a spanning $r$-regular subgraph of $G$.
If $H$ is an $r$-factor of $G$ and $r$ is even then we also call $H$ an \emph{even factor} of~$G$.

If $G$ is an undirected graph, we write $\delta(G)$ for the minimum
degree of $G$, $\Delta(G)$ for the maximum degree of $G$ and $d(G)$ for the average degree of~$G$. Whenever
$X,Y\subset V(G)$, we write $e_G(X,Y)$ for the number of all those edges which
have one endvertex in $X$ and the other in~$Y$. We write $e_G(X)$ for the number of edges in $G[X]$, and
$e'_G(X,Y):=e_G(X,Y)+e_G(X\cap Y)$. Thus $e'_G(X,Y)$ is the number of ordered pairs $(x,y)$ of vertices
such that $x\in X$, $y\in Y$ and $xy\in E(G)$. Given a vertex $x$ of
$G$, we write $d_G(x)$ for the degree of $x$ in~$G$. We often omit the subscript $G$
if this is unambiguous. Also, if $A\subseteq V(G)$ and
the graph $G$ is clear from the context, we sometimes
write $d_A(x)$ for the number of neighbours of $x$ in $A$.
If $G$ is a digraph, we write $\delta^{+}(G)$ for the minimum outdegree
of $G$ and $\delta^{-}(G)$ for the minimum indegree of $G$. 

In order to simplify the presentation, we omit floors and ceilings
and treat large numbers as integers whenever this does not affect
the argument. The constants in the hierarchies used to state our results
have to be chosen from right to left. More precisely, if we claim
that a result holds whenever $0<1/n\ll a\ll b\ll c\le1$ (where $n$
is the order of the graph or digraph), then this means that there
are non-decreasing functions $f:(0,1]\rightarrow(0,1]$, $g:(0,1]\rightarrow(0,1]$
and $h:(0,1]\rightarrow(0,1]$ such that the result holds for all
$0<a,b,c\le1$ and all $n\in\mathbb{N}$ with $b\le f(c)$, $a\le g(b)$
and $1/n\le h(a)$. We will not calculate these functions explicitly.
Hierarchies with more constants are defined in a similar way.

Whenever $x\in\mathbb{R}$ we shall write $x_{+}:=\max\{x,0\}$. We
will write $a=x\pm\epsilon$ as shorthand for $x-\epsilon\le a\le x+\epsilon$,
and $a\ne x\pm\epsilon$ as shorthand for the statement that either
$a<x-\epsilon$ or $a>x+\epsilon$.

\section{Proof outline and further notation}\label{sec:outline}
\subsection{The extremal graph}\label{sec:extexample}
We start by defining a graph $G_{n,\delta,\textnormal{ext}}$ on $n$ vertices which is
extremal for Theorem~\ref{thm:regbounds} in the sense that
$G_{n,\delta,\textnormal{ext}}$ has minimum degree $\delta$ but the largest degree of an even factor of $G_{n,\delta,\textnormal{ext}}$
is at most the right hand side of~(\ref{regbound}). 
Given $\delta>n/2$, let $\Delta$ be the smallest integer
such that $\Delta(\delta+\Delta-n)$ is even and $\Delta\ge (n+\sqrt{n(2\delta-n)})/2$.
Partition the vertex set of $G_{n,\delta,\textnormal{ext}}$ into two classes
$A$ and $B$, with $|B|=\Delta$ and $|A|=n-\Delta$. Let $G_{n,\delta,\textnormal{ext}}[A]$ be
empty, let $G_{n,\delta,\textnormal{ext}}[B]$ be any $(\delta+\Delta-n)$-regular graph, and let $G_{n,\delta,\textnormal{ext}}[A,B]$
be the complete bipartite graph. Clearly $\delta(G_{n,\delta,\textnormal{ext}})=\delta$.
Moreover, if $H$ is a factor of $G_{n,\delta,\textnormal{ext}}$,
then  one can show that $d(H)$ is at most the right hand side of~(\ref{regbound})
(see~\cite{Hartkefactors} for details). In particular, $G_{n,\delta,\textnormal{ext}}$ contains at most $d(H)/2$
Hamilton cycles. Essentially the same construction was given in~\cite{CKO}.

\subsection{Tools and proof overview} \label{toolsproof}
An important concept in our proofs of Theorems~\ref{thm:mainresult} and~\ref{thm:halfnresult}
will be the notion of robust expanders.
This concept was first introduced by K\"uhn, Osthus and Treglown~\cite{KOTchvatal} for directed graphs.
Roughly speaking, a graph is a robust expander if for every set
$S$ which is not too small and not too large, its ``robust'' neighbourhood is at
least a little larger than $S$.
\begin{defn}
Let $G$ be a graph on $n$ vertices.
Given $0<\nu\le\tau<1$ and $S\subseteq V(G)$, we define the \emph{$\nu$-robust neighbourhood}
$RN_{\nu,G}(S)$ of $S$ to be the set of all vertices $v\in V(G)$
with $d_{S}(v)\ge\nu n$. We say that $G$ is a \emph{robust $(\nu,\tau)$-expander}
if for all $S\subset V(G)$ with $\tau n\le|S|\le(1-\tau)n$, we have
$|RN_{\nu,G}(S)|\ge|S|+\nu n$.
\end{defn}
The main tool for our proofs is the following result of K\"uhn and
Osthus~\cite{Kelly} which states that every even-regular robust expander $G$ whose degree is linear in~$|G|$
has a Hamilton decomposition. 
\begin{thm}\label{thm:hamdecresult}
For every $\alpha>0$, there exists $\tau>0$
such that for every $\nu>0$, there exists $n_{0}(\alpha,\tau,\nu)$ such
that the following holds. Suppose that 
\begin{enumerate}
\item $G$ is an $r$-regular graph on $n\ge n_{0}$ vertices, where $r\ge\alpha n$
and $r$ is even; 
\item $G$ is a robust $(\nu,\tau)$-expander. 
\end{enumerate}
Then $G$ has a Hamilton decomposition.
\end{thm}
Let $G$ be a graph on $n$ vertices as in Theorem~\ref{thm:mainresult}.
Let $\delta:=\delta(G)=(1/2+\alpha)n$. (So $\alpha\ge \eps$.) As observed in~\cite{Kellyapps},
every graph on $n$ vertices whose minimum degree is at least slightly larger than $n/2$ is
a robust expander (see Lemma~\ref{lem:anclem1}). Thus our given graph $G$ is a robust expander.
Let $G^*$ be an even factor of largest degree in $G$. So $d(G^*)\ge \textnormal{reg}_{\textnormal{even}}(n,\delta)$.
If $G^*$ is still a robust expander, then we can apply Theorem~\ref{thm:hamdecresult}
to obtain a Hamilton decomposition of $G^*$ and thus at least $\textnormal{reg}_{\textnormal{even}}(n,\delta)/2$
edge-disjoint Hamilton cycles in~$G$. The problem is that if $\alpha$ is small, then
we could have $d(G^*)\le n/2$. So we cannot guarantee that $G^*$ is a robust expander.
(However, this approach works if $\alpha$ is at least slightly larger than $3/2-\sqrt{2}$.
Indeed, in this case Theorem~\ref{thm:regbounds} implies that $d(G^*)$ will be
slightly larger than $n/2$ and so $G^*$ will be a robust expander. This observation was used in~\cite{Kellyapps}
to prove Theorem~\ref{thm:mainresult} for such values of $\alpha$.)

So instead of using this simple strategy, in the proof of Theorem~\ref{thm:mainresult} we will distinguish two cases
depending on whether our graph $G$ contains a subgraph which is close to $G_{n,\delta,\textnormal{ext}}$.
Suppose first that $G$ contains such a subgraph,
$G_{1}$ say. We can choose $G_{1}$ in such a way that $\delta(G_{1})=\delta$,
so $G_{1}$ must have an even factor $G_{2}$ of degree at least $\textnormal{reg}_{\textnormal{even}}(n,\delta)$.
We will then use the fact that $G_1$ is close to $G_{n,\delta,\textnormal{ext}}$
in order to prove directly that $G_{2}$ is a robust expander. As before, this yields
a Hamilton decomposition of $G_{2}$ by Theorem~\ref{thm:hamdecresult}.
This part of the argument is contained in Section~\ref{sec:near}.

If $G$ does not contain a subgraph close to $G_{n,\delta,\textnormal{ext}}$,
then we will first find a sparse even factor $H$ of $G$ which
is still a robust expander and remove it from $G$. Call the resulting graph $G'$.
We will then use the fact that $G$ is far from containing $G_{n,\delta,\textnormal{ext}}$
to show that $G'$ still contains an even factor $H'$ of degree at
least $\textnormal{reg}_{\textnormal{even}}(n,\delta)$. Since robust expansion
is a monotone property, it follows that $H+H'$ is still a robust expander and may therefore
be decomposed into Hamilton cycles by Theorem~\ref{thm:hamdecresult}. So in this case we
even find slightly more than $\textnormal{reg}_{\textnormal{even}}(n,\delta)/2$
edge-disjoint Hamilton cycles. This part of the argument is contained in
Section~\ref{sec:far}. Altogether this will
then imply Theorem~\ref{thm:mainresult}.

In order to prove Theorem~\ref{thm:halfnresult}, we first show that every graph $G$ whose minimum degree is close to $n/2$ either satisfies conditions (i) and (ii) 
of Theorem~\ref{thm:halfnresult} or is a robust expander which does not contain a subgraph close to $G_{n,\delta,\textnormal{ext}}$. So suppose $G$ does not satisfy (i) and (ii). 
We will use the fact that $G$ is a robust expander to find a sparse robustly expanding even factor of $G$, and then argue similarly to the 
second part of the proof of Theorem~\ref{thm:mainresult} to find slightly more than
$\textnormal{reg}_{\textnormal{even}}(n,\delta)/2$ edge-disjoint Hamilton cycles in $G$. This proof is contained in Section~\ref{sec:nhalf}.

\subsection{$\eta$-extremal graphs}
The following definition formalises the notion of {}``containing
a subgraph close to $G_{n,\delta,\textnormal{ext}}$''. For technical
reasons we extend the definition to the case where $\alpha$ is negative
-- this will be used in Section~\ref{sec:nhalf} (with $|\alpha|$ very small). Note
that if $\delta=(1/2+\alpha)n$, then the vertex classes $A$ and
$B$ of $G_{n,\delta,\textnormal{ext}}$ have sizes roughly $(1/2-\sqrt{\alpha/2})n$
and $(1/2+\sqrt{\alpha/2})n$ respectively, and that $G_{n,\delta,\textnormal{ext}}[B]$
is regular of degree roughly $(\alpha+\sqrt{\alpha/2})n$. 
\begin{defn}
\label{def:extremal}Let $\eta>0$ and $-1/2\le\alpha\le1/2$, and
let $G$ be a graph on $n$ vertices with $\delta(G)=\left(1/2+\alpha\right)n$.
Recall that $\alpha_{+}=\max\{\alpha,0\}$. We say that $G$ is \emph{$\eta$-extremal}
if there exist disjoint $A,B\subset V(G)$ such that 
\begin{enumerate}
\item [(E1)]$|A|=(1/2-\sqrt{\alpha_{+}/2}\pm\eta)n$;
\item [(E2)]$|B|=(1/2+\sqrt{\alpha_{+}/2}\pm\eta)n$;
\item [(E3)]$e(A,B)>(1-\eta)|A||B|$; 
\item [(E4)]$e(B)<(\alpha_{+}+\sqrt{\alpha_{+}/2}+\eta)n|B|/2$.
\end{enumerate}
\end{defn}
Note that (E1) and (E2) together imply
\begin{enumerate}
\item [(E5)]$n-|A\cup B|\le2\eta n$.
\end{enumerate}
The following result states that if $G$ is $\eta$-extremal, then
$G[B]$ is {}``almost regular''.
\begin{lem}
\label{lem:almostregular}Suppose $0<\eta\ll\alpha,1/2-\alpha<1/2$.
Suppose $G$ is an $\eta$-extremal graph on $n$ vertices, with $\delta(G)=\left(1/2+\alpha\right)n$,
and let $A,B\subset V(G)$ be as in Definition~\ref{def:extremal}.
\begin{enumerate}
\item For all vertices $v\in B$, we have $d_{B}(v)\ge(\alpha+\sqrt{\alpha/2}-3\eta)n$.
\item For all but at most $2\sqrt{\eta}n$ vertices $v\in B$, we have $d_{B}(v)\le(\alpha+\sqrt{\alpha/2}+2\sqrt{\eta})n$.
\end{enumerate}
\end{lem}
\begin{proof}
(i) immediately follows from (E1) and (E5). Indeed, for all $v\in B$,
we have\begin{eqnarray}
d_{B}(v) & \ge & \delta(G)-d_{A}(v)-d_{\overline{A\cup B}}(v)\stackrel{\textnormal{(E5)}}{\ge}\delta(G)-|A|-2\eta n\nonumber \\
 & \stackrel{\textnormal{(E1)}}{\ge} & \left(\alpha+\sqrt{\frac{\alpha}{2}}-3\eta\right)n,\label{eq:(*)}\end{eqnarray}
as desired. 

Suppose (ii) fails. Then there exist at least $2\sqrt{\eta}n$ vertices
in $B$ with degree greater than $(\alpha+\sqrt{\alpha/2}+2\sqrt{\eta})n$
in $B$. We therefore have\begin{align*}
e_{G}(B) & =\frac{1}{2}\sum_{v\in B}d_{B}(v)\stackrel{(\ref{eq:(*)})}{>}\frac{1}{2}\left(\left(\alpha+\sqrt{\frac{\alpha}{2}}-3\eta\right)n|B|+2\sqrt{\eta}n\cdot2\sqrt{\eta}n\right)\\
 & \ge\frac{1}{2}\left(\alpha+\sqrt{\frac{\alpha}{2}}+\eta\right)n|B|.\end{align*}
But this contradicts (E4), so (ii) must hold.
\end{proof}

\section{The near-extremal case}\label{sec:near}

Suppose that $0<1/n\ll\eta\ll\alpha<1/2$, and that $G$ is an $\eta$-extremal
graph on $n$ vertices with $\delta(G)=(1/2+\alpha)n$. Recall that
our aim in this case is to show that $G$ contains a factor of degree $\textnormal{reg}_{\textnormal{even}}(n,\delta)/2$
which is a robust expander.
Let $A,B\subset V(G)$ be as in Definition~\ref{def:extremal}. We will
first show that $G$ contains a spanning subgraph $G_{1}$ which
is close to $G_{n,\delta,\textnormal{ext}}$ and satisfies $\delta(G_{1})=\delta(G)$.
\begin{lem}\label{lem:closesubgraph}
Suppose $0<1/n\ll\eta\ll1/C\ll1/2-\alpha\le1/2$,
so that in particular $0\le\alpha<1/2$. Let $G$ be an $\eta$-extremal
graph on $n$ vertices with $\delta:=\delta(G)=\left(1/2+\alpha\right)n$,
and let $A,B\subset V(G)$ be as in Definition~\ref{def:extremal}. Then
there exists a spanning subgraph $G_{1}$ of $G$ which satisfies the following properties: 
\begin{enumerate}
\item $A$ and $B$ satisfy (E1)--(E4) for the graph $G_{1}$. In particular,
$G_{1}$ is $\eta$-extremal.
\item $\delta(G_{1})=\delta$.
\item $e_{G_{1}}(A)<C\eta|A|^{2}$.
\end{enumerate}
\end{lem}
\begin{proof}
We will define $G_{1}$ using a greedy algorithm. Initially, let $G_{1}:=G$.
Suppose that $G_{1}[A]$ contains an edge $xy$ such that $d_{G_{1}}(x),d_{G_{1}}(y)>\delta$.
Then remove $xy$ from $G_{1}$, and continue in this way until $G_{1}$
contains no such edge. Note that we have $\delta(G_{1})=\delta$,
and (E1)--(E4) are not affected by these edge deletions, so $G_{1}$
satisfies (i) and (ii). 

Suppose $e_{G_{1}}(A)\ge C\eta|A|^{2}$, and note that we have\[
\delta=\left(\frac{1}{2}+\alpha\right)n\le\left(\frac{1}{2}+\sqrt{\frac{\alpha}{2}}\right)n\stackrel{\textnormal{(E2)}}{\le}|B|+\eta n.\]
(Indeed, $x\le\sqrt{x/2}$ for all $0\le x\le1/2$.) If $v\in A$
is a vertex with $d_{G_{1}}(v)=\delta$, we therefore have\[
d_{G[A,B]}(v)=d_{G_{1}[A,B]}(v)\le \delta-d_{G_{1}[A]}(v)\le|B|+\eta n-d_{G_{1}[A]}(v).\]
Each edge in $G_{1}[A]$ must have at least one endpoint with degree
$\delta$ in $G_{1}$, so%
\COMMENT{Note that the second inequality need not be an equality, e.g. if an
edge in $A$ has two endpoints of degree $\delta$.}
\begin{align*}
e_{G}(A,B) & =\sum_{v\in A}d_{G[A,B]}(v)\le|A||B|-\sum_{v\in A,\, d_{G_{1}}(v)=\delta}\left(d_{G_{1}[A]}(v)-\eta n\right)\\
 & \le|A||B|+\eta n^{2}-e_{G_{1}}(A)\le|A|\left(|B|+\eta\frac{n^{2}}{|A|}-C\eta|A|\right).\end{align*}
Since $1/C\ll1/2-\alpha$, we have $C|A|\ge2|B|+n^{2}/|A|$
by (E1) and (E2). Hence\[
e_{G}(A,B)\le|A|(|B|-2\eta|B|)=(1-2\eta)|A||B|,\]
which contradicts (E3). We therefore have $e_{G_1}(A)<C\eta|A|^{2}$,
and so $G_{1}$ satisfies (iii) as desired.
\end{proof}
Let $G_{1}$ be as in Lemma~\ref{lem:closesubgraph}, and let $G_{2}$
be a degree-maximal even factor of $G_{1}$. (So $G_{2}$ is an even-regular
spanning subgraph of $G_1$ whose degree is as large as possible.) By Theorem~\ref{thm:regbounds},
we have that\begin{equation}
d(G_{2})\ge\textnormal{reg}_{\textnormal{even}}(n,\delta)\ge\frac{n}{4}+\frac{\alpha n}{2}+\sqrt{\frac{\alpha}{2}}n-2.\label{eq:xx}\end{equation}
It can be shown that any degree-maximal even factor of $G_{n,\delta,{\rm ext}}$ contains almost all edges inside
the larger vertex class~$B$.%
    \COMMENT{TO DO: check}
The following lemma uses a similar argument to prove a similar statement for~$G_1$.
\begin{lem}
\label{lem:stillalmostregular}Suppose $0<1/n\ll\eta\ll1/C\ll\alpha,1/2-\alpha<1/2$.
Suppose that $G$ is an $\eta$-extremal graph on $n$ vertices with $\delta(G)=(1/2+\alpha)n$.
Let $G_{1}$ be the graph obtained by applying Lemma~\ref{lem:closesubgraph}
to $G$, and let $G_{2}$ be a degree-maximal even factor of $G_{1}$. Let
$A,B\subset V(G)$ be as in Definition~\ref{def:extremal}. Then for all
but at most $3\eta^{1/4}n$ vertices $v\in B$, we have
\[
d_{G_{2}[B]}(v)\ge\left(\alpha+\sqrt{\frac{\alpha}{2}}-3\eta^{\frac{1}{4}}\right)n.\]
\end{lem}
\begin{proof}
Let $r$ be the degree of $G_{2}$. Suppose that $d_{G_{2}[B]}(v)<(\alpha+\sqrt{\alpha/2}-3\eta^{1/4})n$
for more than $3\eta^{1/4}n$ vertices. Then by Lemma~\ref{lem:almostregular}(ii),
we have\begin{align*}
r|B| & =\sum_{v\in B}d_{G_{2}}(v)=e_{G_{2}}(A,B)+2e_{G_{2}}(B)\\
 & \le r|A|+\left(\alpha+\sqrt{\frac{\alpha}{2}}+2\sqrt{\eta}\right)n|B|+4\sqrt{\eta} n^{2}-3\eta^{\frac{1}{4}}n\cdot 3\eta^{\frac{1}{4}}n\\
 & \le r|A|+\left(\alpha+\sqrt{\frac{\alpha}{2}}-3\sqrt{\eta}\right)n|B|.\end{align*}
Since $|B|-|A|\ge(\sqrt{2\alpha}-2\eta)n$ by (E1) and (E2), it follows
that\[
\sqrt{2\alpha}rn\le\left(\alpha+\sqrt{\frac{\alpha}{2}}-3\sqrt{\eta}\right)n|B|+2\eta n^{2},\]
and hence\begin{eqnarray*}
r & \le & \left(\sqrt{\frac{\alpha}{2}}+\frac{1}{2}-3\sqrt{\frac{\eta}{2\alpha}}\right)|B|+\eta\sqrt{\frac{2}{\alpha}}n\\
 & \stackrel{\textnormal{(E2)}}{\le} & \left(\sqrt{\frac{\alpha}{2}}+\frac{1}{2}-3\sqrt{\frac{\eta}{2\alpha}}\right)
\left(\frac{1}{2}+\sqrt{\frac{\alpha}{2}}+\eta\right)n+\eta\sqrt{\frac{2}{\alpha}}n\\
 & \le & \left(\frac{1}{4}+\frac{\alpha}{2}+\sqrt{\frac{\alpha}{2}}-\eta^{3/4}\right)n.\end{eqnarray*}
(In the last inequality we used that $\eta\ll \alpha$.)
It therefore follows from (\ref{eq:xx}) that $r<{\rm reg}_{\rm even}(n,\delta)$.
But $G_{2}$ was chosen to be degree-maximal, so this is a contradiction.
\end{proof}
We now collect some robust expansion properties of $G_{2}$. For convenience,
if $X\subset V(G_{2})$, we shall write $X_{A}:=X\cap A$ and $X_{B}:=X\cap B$.
In particular, if $S\subset V(G)$ then (for example) $RN_{\nu}(S_{A})_{B}=RN_{\nu}(S\cap A)\cap B$.
\begin{lem}
\label{lem:expansionprops}Suppose that $0<1/n\ll\nu\ll\eta\ll\mu\ll\tau\ll\lambda\ll1/C\ll\alpha,1/2-\alpha<1/2$.
Suppose that $G$ is an $\eta$-extremal graph on $n$ vertices with $\delta(G)=\left(1/2+\alpha\right)n$.
Let $G_{1}$ be the graph obtained by applying Lemma~\ref{lem:closesubgraph}
to $G$, and let $G_{2}$ be a degree-maximal even factor of $G_{1}$. Let
$A,B\subset V(G)$ be as in Definition~\ref{def:extremal}. Then in the
graph $G_{2}$, the following properties all hold.
\begin{enumerate}
\item If $S\subset A$ with $|S|\ge|A|/2$, then $|RN_{\nu}(S)_{B}|\ge(1-\mu)|B|$.
\item If $S\subset B$ with $|S|\ge|B|/2$, then $|RN_{\nu}(S)_{A}|\ge(1-\mu)|A|$.
\item If $S\subset A$ with $|S|\ge \tau n/3$, then $|RN_{\nu}(S)_{B}|\ge|B|/2+\lambda n$.
\item If $S\subset B$ with $|S|\ge \tau n/3$, then $|RN_{\nu}(S)_{A}|\ge|A|/2+\lambda n$.
\item If $S\subset B$, then $|RN_{\nu}(S)_{B}|\ge|S|-\mu n$.
\end{enumerate}
\end{lem}
\begin{proof}
Write $d:=d(G_{2})$. We first prove (i). Suppose $S\subset A$ with
$|S|\ge|A|/2$. Lemma~\ref{lem:almostregular}(ii) implies that in $G_2$
all but at most $2\sqrt{\eta}n\le\mu |B|$ vertices $v\in B$
satisfy\begin{eqnarray*}
d_{A}(v) & = & d-d_{\overline{A\cup B}}(v)-d_{B}(v)\\
 & \stackrel{\textnormal{(\ref{eq:xx})},\textnormal{(E5)}}{\ge} & \left(\frac{1}{4}+\frac{\alpha}{2}+\sqrt{\frac{\alpha}{2}}\right)n-2-2\eta n-\left(\alpha+\sqrt{\frac{\alpha}{2}}+2\sqrt{\eta}\right)n\\
 & \ge & \left(\frac{1}{4}-\frac{\alpha}{2}-3\sqrt{\eta}\right)n\ge\left(\frac{1}{4}-\frac{1}{2}\sqrt{\frac{\alpha}{2}}+\eta\right)n+\nu n\stackrel{\textnormal{(E1)}}{\ge}\frac{|A|}{2}+\nu n,\end{eqnarray*}
where the third inequality follows since $x<\sqrt{x}/2$ for all $0<x<1/4$.
Thus in the graph $G_{2}$ we have $|N_{A}(v)\cap S|\ge\nu n$, and
hence $v\in RN_{\nu}(S)$, for each such $v$. The result therefore
follows.

We now prove (ii). Suppose $S\subset B$ with $|S|\ge|B|/2$. Let
$A'\subset A$ be the set of vertices which in $G_2$ have less than
$|B|/2+\nu n$ neighbours inside $B$. Each vertex $v\in A'$ must satisfy\begin{eqnarray*}
d_{A}(v) & = & d-d_{\overline{A\cup B}}(v)-d_{B}(v)\\
 & \stackrel{\textnormal{(\ref{eq:xx})},\textnormal{(E5)}}{\ge} & \left(\frac{1}{4}+\frac{\alpha}{2}+\sqrt{\frac{\alpha}{2}}\right)n-2-2\eta n-\frac{|B|}{2}-\nu n\\
 & \stackrel{\textnormal{(E2)}}{\ge} & \left(\frac{1}{4}+\frac{\alpha}{2}+\sqrt{\frac{\alpha}{2}}-2\eta-\nu\right)n-2-\left(\frac{1}{4}+\frac{1}{2}\sqrt{\frac{\alpha}{2}}+\eta\right)n\\
 & \ge & \frac{\alpha}{2}n,\end{eqnarray*}
and so we have $e_{G_2}(A)\ge\alpha n|A'|/4$. But by Lemma~\ref{lem:closesubgraph}(iii)
we have $e_{G_2}(A)\le e_{G_1}(A)<C\eta|A|^{2}$. Therefore \[
|A'|\le\frac{4C\eta}{\alpha}\cdot\frac{|A|^{2}}{n}\le\sqrt{\eta}\frac{|A|^{2}}{n}\le\sqrt{\eta}|A|\le\mu|A|.\]

However, our assumption that $|S|\ge|B|/2$ and the definition of
$A'$ together imply that every vertex $v\in A\setminus A'$ satisfies $|N_{B}(v)\cap S|\ge\nu n$.
Therefore $|RN_{\nu}(S)|\ge|A\setminus A'|\ge(1-\mu)|A|$, as required.

We now prove (iii). Suppose $S\subset A$ with $|S|\ge \tau n/3$.
Then we double-count the edges between $S$ and $B$ in $G_{2}$.
The definition of a robust neighbourhood implies that
\begin{align*}
e_{G_{2}}(S,B) & =e_{G_{2}}(S,RN_{\nu}(S)_{B})+e_{G_{2}}(S,B\setminus RN_{\nu}(S)_{B})\le|S||RN_{\nu}(S)_{B}|+\nu n^{2}.\end{align*}
On the other hand, Lemma~\ref{lem:closesubgraph}(iii) implies that\begin{align*}
e_{G_{2}}(S,B) & \ge d|S|-2e_{G_{2}}(S,A)-e_{G_{2}}(S,\overline{A\cup B})\stackrel{\textnormal{(E5)}}{\ge}d|S|-2C\eta|A|^{2}-2\eta n^{2}\\
 & \ge d|S|-3C\eta n^{2}.\end{align*}
Combining the two inequalities yields\begin{eqnarray*}
|RN_{\nu}(S)_{B}| & \ge & d-3C\eta\frac{n^{2}}{|S|}-\nu\frac{n^{2}}{|S|}\\
 & \stackrel{(\ref{eq:xx})}{\ge} & \left(\frac{1}{4}+\frac{\alpha}{2}+\sqrt{\frac{\alpha}{2}}\right)n-2-\frac{9C\eta}{\tau}n-\frac{3\nu}{\tau}n\\
 & \stackrel{\textnormal{(E2)}}{\ge} & \frac{|B|}{2}+\left(\frac{\alpha}{2}+\frac{1}{2}\sqrt{\frac{\alpha}{2}}-\eta-\frac{9C\eta}{\tau}-\frac{3\nu}{\tau}\right)n-2\ge\frac{|B|}{2}+\frac{\alpha}{2}n,\end{eqnarray*}
and so the result follows.

We now prove (iv). Suppose $S\subset B$ with $|S|\ge \tau n/3$.
Then we double-count the edges between $S$ and $A$ in $G_{2}$.
As before, we have \begin{equation}
e_{G_{2}}(S,A)\le|S||RN_{\nu}(S)_{A}|+\nu n^{2}.\label{eq:***}\end{equation}
On the other hand, \[
e_{G_{2}}(S,A)\ge d|S|-\sum_{v\in S}d_{B}(v)-\sum_{v\in S}d_{\overline{A\cup B}}(v)\stackrel{\textnormal{(E5)}}{\ge}
d|S|-\sum_{v\in S}d_{B}(v)-2\eta n^{2}.\]
Lemma~\ref{lem:almostregular}(ii) implies that\[
\sum_{v\in S}d_{B}(v)\le2\sqrt{\eta}n^{2}+\left(\alpha+\sqrt{\frac{\alpha}{2}}+2\sqrt{\eta}\right)n|S|,\]
and so\begin{eqnarray*}
e_{G_2}(S,A) & \stackrel{\textnormal{(\ref{eq:xx})},\textnormal{(E5)}}{\ge} & \left(\frac{1}{4}+\frac{\alpha}{2}+\sqrt{\frac{\alpha}{2}}\right)n|S|-2|S|-\left(\alpha+\sqrt{\frac{\alpha}{2}}+2\sqrt{\eta}\right)n|S|-(2\eta+2\sqrt{\eta})n^{2}\\
 & \ge & \left(\frac{1}{4}-\frac{\alpha}{2}\right)n|S|-5\sqrt{\eta}n^{2}.\end{eqnarray*}
Combining this with (\ref{eq:***}) shows that in $G_{2}$ we
have\begin{align*}
|RN_{\nu}(S)_{A}| & \ge\left(\frac{1}{4}-\frac{\alpha}{2}\right)n-6\sqrt{\eta}\cdot\frac{n^{2}}{|S|}\ge\left(\frac{1}{4}-\frac{\alpha}{2}\right)n-\frac{18\sqrt{\eta}}{\tau}n\\
 & =\left(\frac{1}{4}-\frac{1}{2}\sqrt{\frac{\alpha}{2}}\right)n+\left(\frac{1}{2}\sqrt{\frac{\alpha}{2}}-\frac{\alpha}{2}\right)n-\frac{18\sqrt{\eta}}{\tau}n\\
 & \stackrel{\textnormal{(E1)}}{\ge}\frac{|A|}{2}+\frac{1}{2}\left(\frac{1}{2}\sqrt{\frac{\alpha}{2}}-
\frac{\alpha}{2}\right)n\ge\frac{|A|}{2}+\lambda n,\end{align*}
and so the result follows. (Here we used that $\sqrt{x}/2>x$ for
all $0<x<1/4$.)

Finally, we prove (v). Suppose $S\subset B$. Recall that $e'_{G_2}(S,RN_{\nu}(S)_{B})$ denotes
the number of ordered pairs $(u,v)$ of vertices of $G_{2}$ such
that $uv\in E(G_{2})$, $u\in S$ and $v\in RN_{\nu}(S)_{B}$. (Note
that this may not equal $e(S,RN_{\nu}(S)_{B})$, as we may have $S\cap RN_{\nu}(S)_{B}\ne\emptyset$.)
By Lemma~\ref{lem:almostregular}(ii), counting from $RN_{\nu}(S)_{B}$
yields\[
e'_{G_2}(S,RN_{\nu}(S)_{B})\le\left(\alpha+\sqrt{\frac{\alpha}{2}}+2\sqrt{\eta}\right)n|RN_{\nu}(S)_{B}|+2\sqrt{\eta}n^{2}.\]
By Lemma~\ref{lem:stillalmostregular}, counting from $S$ yields
\begin{align*}
e'_{G_2}(S,RN_{\nu}(S)_{B}) & \ge\left(\alpha+\sqrt{\frac{\alpha}{2}}-3\eta^{\frac{1}{4}}\right)n|S|-3\eta^{\frac{1}{4}}n^{2}-\nu n^{2}.\end{align*}
Combining the two inequalities yields $|RN_{\nu}(S)_{B}|\ge|S|-\mu n$
as desired.
\end{proof}
We are now in a position to prove Theorem~\ref{thm:mainresult} for $\eta$-extremal graphs.
\begin{lem}
\label{lem:mainresultpart1}Suppose $0<1/n\ll\eta\ll\alpha,1/2-\alpha<1/2$.
If $G$ is an $\eta$-extremal graph on $n$ vertices with $\delta:=\delta(G)=(1/2+\alpha)n$,
then $G$ contains at least $\textnormal{reg}_{\textnormal{even}}(n,\delta)/2$
edge-disjoint Hamilton cycles.\end{lem}
\begin{proof}
Let $\tau_{0}:=\tau(1/4)$ be the constant returned by Theorem~\ref{thm:hamdecresult}.
Choose additional constants $\nu,\mu,\tau,\lambda$ and $C$ such that
\[
0<\nu\ll\eta\ll\mu\ll\tau\ll\lambda\ll1/C\ll\alpha,1/2-\alpha,\tau_{0}.\]
Take $A,B\subset V(G)$ as in Definition~\ref{def:extremal}. Apply Lemma~\ref{lem:closesubgraph}
to $G$ and $C$ to obtain a spanning subgraph $G_{1}$. Let $G_{2}$
be a degree-maximal even factor of $G_{1}$. Note that Lemma~\ref{lem:expansionprops}
may be applied to $G_{2}$. 

\medskip

\noindent\textbf{Claim: }\emph{$G_{2}$ is a robust $(\nu,\tau)$-expander.} 

Note that the claim immediately implies the desired result. Indeed,
any robust $(\nu,\tau)$-expander is also a robust $(\nu,\tau_{0})$-expander,
and so Theorem~\ref{thm:hamdecresult} implies that $G_{2}$ may be
decomposed into Hamilton cycles. Moreover, Lemma~\ref{lem:closesubgraph}
implies that $\delta(G_{1})=\delta$ and so $d(G_{2})\ge\textnormal{reg}_{\textnormal{even}}(n,\delta)$.
Hence the Hamilton decomposition of $G_{2}$ yields the desired collection
of $d(G_{2})/2\ge\textnormal{reg}_{\textnormal{even}}(n,\delta)/2$
edge-disjoint Hamilton cycles.

To prove the claim, consider $S\subset V(G)$ with $\tau n\le|S|\le(1-\tau)n$.
We will use Lemma~\ref{lem:expansionprops} to show that in $G_2$ we have $|RN_{\nu}(S)|\ge|S|+\nu n$.
We will split the proof into cases depending on the sizes of $S_{A}=S\cap A$
and $S_{B}=S\cap B$. Note that $|S_{\overline{A\cup B}}|\le 2\eta n$ by~(E5).

\medskip

\noindent\textbf{Case 1:} $|S_{A}|\le\tau n/3$, $|S_{B}|\le\tau n/3$. 

In this case, we have\[
|S|\stackrel{\textnormal{(E5)}}{\le}\frac{2\tau}{3}+2\eta n<\tau n,\]
which is a contradiction.

\medskip

\noindent\textbf{Case 2:} $\tau n/3\le|S_{A}|\le|A|/2$, $|S_{B}|\le\tau n/3$. 

In this case, by Lemma~\ref{lem:expansionprops}(iii) we have\begin{align*}
|RN_{\nu}(S)| & \ge|RN_{\nu}(S_{A})_{B}|\ge\frac{|B|}{2}+\lambda n
\stackrel{\textnormal{(E5)}}{\ge}\frac{|A|}{2}+\sqrt{\frac{\alpha}{2}}n-2\eta n+\lambda n\\
 & \ge\frac{|A|}{2}+\frac{\tau}{3}n+2\eta n+\nu n\ge|S|+\nu n,\end{align*}
as desired.

\medskip

\noindent\textbf{Case 3:} $|S_{A}|\ge|A|/2$, $|S_{B}|\le\tau n/3$. 

In this case, by Lemma~\ref{lem:expansionprops}(i) we have\begin{align*}
|RN_{\nu}(S)| & \ge|RN_{\nu}(S_{A})_{B}|\ge(1-\mu)|B|\ge|A|+2\sqrt{\frac{\alpha}{2}}n-2\eta n-\mu n\\
 & \ge|A|+\frac{\tau}{3}n+2\eta n+\nu n\ge|S|+\nu n,\end{align*}
as desired.

\medskip

\noindent\textbf{Case 4:} $|S_{A}|\le|A|/2$, $|S_{B}|\ge\tau n/3$. 

In this case, by Lemma~\ref{lem:expansionprops}(iv) and (v), we have\begin{align*}
|RN_{\nu}(S)| & \ge|RN_{\nu}(S_{B})_{A}|+|RN_{\nu}(S_{B})_{B}|\ge\frac{|A|}{2}+\lambda n+|S_{B}|-\mu n\\
 & \ge|S_{A}|+|S_{B}|+2\eta n+\nu n\ge |S|+\nu n,\end{align*}
as desired.

\medskip

\noindent\textbf{Case 5:} $|S_{A}|\ge|A|/2$, $\tau n/3\le|S_{B}|\le|B|/2$. 

In this case, by Lemma~\ref{lem:expansionprops}(i) and (iv), we have\begin{align*}
|RN_{\nu}(S)| & \ge|RN_{\nu}(S_{A})_{B}|+|RN_{\nu}(S_{B})_{A}|\ge|B|+\frac{|A|}{2}+(\lambda-\mu)n\\
 & \ge\frac{|B|}{2}+|A|+(\lambda-\mu)n\ge|S_{B}|+|S_{A}|+2\eta n+\nu n\ge |S|+\nu n,\end{align*}
as desired, where the third inequality follows since $|B|\ge|A|$ by~(E1) and~(E2).

\medskip

\noindent\textbf{Case 6:} $|S_{A}|\ge|A|/2$, $|S_{B}|\ge|B|/2$. 

In this case, by Lemma~\ref{lem:expansionprops}(i) and (ii), we have\begin{eqnarray*}
|RN_{\nu}(S)| & \ge & |RN_{\nu}(S_{A})_{B}|+|RN_{\nu}(S_{B})_{A}|\ge|B|+|A|-2\mu n\\
 & \stackrel{\textnormal{(E5)}}{\ge} & n-(2\eta+2\mu)n\ge(1-\tau)n+\nu n\ge|S|+\nu n,\end{eqnarray*}
as desired.

Thus in all cases we have $|RN_{\nu}(S)|\ge|S|+\nu n$. Indeed, if
$|S_{B}|\le\tau n/3$ this follows by Cases 1, 2 and 3; if $\tau n/3\le|S_{B}|\le|B|/2$
this follows by Cases 4 and 5; and if $|S_{B}|\ge|B|/2$ this follows
by Cases 4 and 6. Hence $G_{2}$ is a robust $(\nu,\tau)$-expander
as desired. This proves the claim and hence the lemma.
\end{proof}

\section{The non-extremal case}\label{sec:far}

Suppose now that $G$ is not $\eta$-extremal. Our first aim is to
find a sparse even factor $H$ of $G$ which is a robust expander. This
has essentially already been done in \cite{Kellyapps}, but for digraphs.
We first require the following definition,
which generalises the notion of robust expanders to digraphs.
\begin{defn}
\label{def:robustoutexpander}Let $D$ be a digraph on $n$ vertices.
Given $0<\nu\le\tau<1$, we define the \emph{$\nu$-robust outneighbourhood}
$RN_{\nu,D}^{+}(S)$ of $S$ to be the set of all vertices $v\in V(D)$
which have at least $\nu n$ inneighbours in $S$. We say that $D$
is a \emph{robust $(\nu,\tau)$-outexpander} if for all $S\subset V(D)$
with $\tau n\le|S|\le(1-\tau)n$, we have $|RN_{\nu,D}^{+}(S)|\ge|S|+\nu n$.
\end{defn}
We will now quote three lemmas from \cite{Kellyapps}. Lemma~\ref{lem:anclem1} implies that
our given graph $G$ is a robust expander. We will use Lemmas~\ref{lem:anclem2} and~\ref{lem:anclem3}
to deduce Lemma~\ref{lem:smallexpanderfactor}, which together with Lemma~\ref{lem:anclem1} implies that
$G$ contains a sparse even factor $H$ which is still a robust expander.
\begin{lem}
\label{lem:anclem1}Suppose $0<\nu\le\tau\le\epsilon<1/2$ and $\epsilon\ge2\nu/\tau$.
Let $G$ be a graph on $n$ vertices with minimum degree $\delta(G)\ge\left(1/2+\epsilon\right)n$.
Then $G$ is a robust $(\nu,\tau)$-expander.
\end{lem}
\begin{lem}
\label{lem:anclem2}Suppose $0<1/n\ll\eta\ll\nu,\tau,\alpha<1$. Suppose
that $G$ is a robust $(\nu,\tau)$-expander on $n$ vertices with
$\delta(G)\ge\alpha n$. Then the edges of $G$ can be oriented in
such a way that the resulting oriented graph $G'$ satisfies the following
properties:
\begin{enumerate}
\item $G'$ is a robust $\left(\nu/4,\tau\right)$-outexpander.
\item $d_{G'}^{+}(x),d_{G'}^{-}(x)=(1\pm\eta)d_{G}(x)/2$ for every vertex
$x$ of $G$.
\end{enumerate}
\end{lem}
An \emph{$r$-factor} of a digraph $G$ is a spanning subdigraph of $G$ in which every vertex has in- and outdegree~$r$.
\begin{lem}
\label{lem:anclem3}Suppose $0<1/n\ll\nu'\ll\xi\ll\nu\le\tau\ll\alpha<1$.
Let $G$ be a robust $(\nu,\tau)$-outexpander on $n$ vertices with
$\delta^{+}(G),\delta^{-}(G)\ge\alpha n$. Then $G$ contains a $\xi n$-factor
which is still a robust $(\nu',\tau)$-outexpander.
\end{lem}

\begin{lem}
\label{lem:smallexpanderfactor}Suppose $0<1/n\ll\nu'\ll\epsilon\ll\nu\ll\tau\ll\alpha<1$,
and suppose in addition that $\epsilon n$ is an even integer. If
$G$ is a robust $(\nu,\tau)$-expander on $n$ vertices with $\delta(G)\ge\alpha n$,
then there exists an $\epsilon n$-factor $H$ of $G$ which is a
robust $(\nu',\tau)$-expander.\end{lem}
\begin{proof}
We apply Lemma~\ref{lem:anclem2} to orient the edges of $G$, forming
an oriented graph $G'$ which is a robust $\left(\nu/4,\tau\right)$-outexpander
and which satisfies $\delta^{+}(G'),\delta^{-}(G')\ge\alpha n/3$.
We then apply Lemma~\ref{lem:anclem3} to find an $\epsilon n/2$-factor
$H$ of $G'$ which is a robust $(\nu',\tau)$-outexpander. Now remove the orientation on the edges of $H$ to
obtain a robust $(\nu',\tau)$-expander which is an $\epsilon n$-factor
of $G$, as desired.
\end{proof}
We will now show that even after removing a sparse factor $H$, our given graph $G$
still contains an even factor of degree at least $\textnormal{reg}_{\textnormal{even}}(n,\delta)$.
To do this, we first show that $G-H$ is still non-extremal.
\begin{lem}
\label{lem:removefactornotextremal}Suppose $0<1/n\ll\epsilon\ll\eta\ll1/2-\alpha$,
and that $-\epsilon\le\alpha<1/2$. Let $G$ be a graph
on $n$ vertices with $\delta(G)=\left(1/2+\alpha\right)n$ which is
not $2\eta$-extremal. Suppose $H$ is an $\epsilon n$-factor of
$G$. Then $G-H$ is not $\eta$-extremal.
\end{lem}
\begin{proof}
Suppose $A,B\subset V(G)$ are disjoint with $|A|$ and $|B|$ satisfying
(E1) and (E2) of Definition~\ref{def:extremal}. Let $G':=G-H$. Since $G$ is not $2\eta$-extremal,
we must have either $e_{G}(A,B)\le(1-2\eta)|A||B|$ or $e_{G}(B)\ge(\alpha_++\sqrt{\alpha_+/2}+2\eta)n|B|/2$.
In the former case we have\[
e_{G'}(A,B)\le e_{G}(A,B)<(1-\eta)|A||B|,\]
and in the latter case we have\begin{eqnarray*}
e_{G'}(B) & \ge & e_{G}(B)-\epsilon n|B|\ge\frac{1}{2}\left(\alpha_++\sqrt{\frac{\alpha_+}{2}}+2\eta-2\epsilon\right)n|B|\\
 & \ge & \frac{1}{2}\left((\alpha-\epsilon)_{+}+\sqrt{\frac{(\alpha-\epsilon)_{+}}{2}}+\frac{3\eta}{2}\right)n|B|.\end{eqnarray*}
Since $\delta(G-H)=\left(1/2+\alpha-\epsilon\right)n$, it follows
that $G-H$ is not $\eta$-extremal.
\end{proof}
We now show that $G-H$ contains a large even factor. We will do this using
the well-known result of Tutte \cite{Tutte}, given below. 
\begin{thm}
\label{thm:Tutte}Let $G$ be a graph. Given disjoint $S,T\subset V(G)$
and $r\in\mathbb{N}$, let $Q_{r}(S,T)$ be the number of connected
components $C$ of $G-(S\cup T)$ such that $r|C|+e(C,T)$ is odd,
and let\begin{equation}
R_{r}(S,T):=\sum_{v\in T}d(v)-e(S,T)+r(|S|-|T|).\label{eq:rdefn}\end{equation}
Then $G$ contains an $r$-factor if and only if $Q_{r}(S,T)\le R_{r}(S,T)$
for all disjoint $S,T\subset V(G)$.
\end{thm}
In proving the following lemma, we follow a similar approach to that
used in \cite{CKO}. We will also make frequent and implicit use of
the inequality $\sqrt{x}\le\sqrt{x+h}\le\sqrt{x}+\sqrt{h}$ for $x,h\ge 0$.
\begin{lem}
\label{lem:largefactor}Suppose $0<1/n\ll\epsilon\ll\eta\ll1/2-\alpha$
and that $-\epsilon\le\alpha<1/2$. Let $G$ be a graph on $n$ vertices
with minimum degree $\delta:=\delta(G)=\left(1/2+\alpha\right)n$, and suppose
that $G$ is not $\eta$-extremal. Let\[
r:=\frac{n}{4}+\frac{(\alpha+\epsilon)n}{2}+\sqrt{\frac{\alpha+\epsilon}{2}}n,\]
and suppose that $r$ is an even integer. Then $G$ contains an $r$-factor.\end{lem}
\begin{proof}
Let $S,T$ be two arbitrary disjoint subsets of $V(G)$. We will show
that $Q_{r}(S,T)\le R_{r}(S,T)$, from which the result follows by
Theorem~\ref{thm:Tutte}. We first note a useful bound on $Q_{r}(S,T)$.
If $\delta\ge |S|+|T|$ then every vertex outside $S\cup T$ has at least $\delta-|S|-|T|$ neighbours
outside $S\cup T$, so every component of $G-(S\cup T)$ contains
at least $\delta-|S|-|T|+1$ vertices. Thus
\begin{equation}\label{eq:qbound}
Q_{r}(S,T)\le\frac{n-|S|-|T|}{\delta-|S|-|T|+1} \ \ \ \ \text{if } \delta\ge |S|+|T|.
\end{equation}
Also, note that we always have
\begin{equation}
\delta-r=\left(\frac{1}{4}+\frac{\alpha-\epsilon}{2}-\sqrt{\frac{\alpha+\epsilon}{2}}\right)n=
\left(\frac{1}{4}+\frac{\alpha+\epsilon}{2}-\sqrt{\frac{\alpha+\epsilon}{2}}-\epsilon\right)n\ge\epsilon n,\label{eq:d-r+ve}
\end{equation}
since $1/4+x-\sqrt{x}=(1/2-\sqrt{x})^{2}>0$ for all $0\le x<1/4$
and since $\epsilon\ll1/2-\alpha$. We will now split the proof into cases
depending on $|S|$ and $|T|$. 

\medskip

\noindent\textbf{Case 1:} $|T|\le r-1$, $|S|\le\delta-r$, and $|S|+|T|\ge3$.

We have\begin{eqnarray}
R_{r}(S,T) & \stackrel{(\ref{eq:rdefn})}{=} & \sum_{v\in T}(d(v)-r)+\sum_{v\in S}(r-d_{T}(v))\ge |T|(\delta-r)+\sum_{v\in S}1\nonumber \\
& \stackrel{(\ref{eq:d-r+ve})}{\ge} & |S|+|T|.\label{eq:rbound}\end{eqnarray}
Let $k:=|S|+|T|$. By (\ref{eq:qbound}) and~(\ref{eq:rbound}) it suffices to show that $k\ge(n-k)/(\delta-k+1)$.
This is equivalent to showing that\[
\delta k-k^{2}+2k-n=(k-2)(\delta-k)+2\delta-n\ge0.\]
We have $3\le k\le\delta-1$ and the function $(k-2)(\delta-k)$ is
concave, so it must be minimised in this range when $k=3$ or when
$k=\delta-1$. In either case, we have\[
(k-2)(\delta-k)+2\delta-n=\delta-3+2\delta-n\ge\delta-3-2\epsilon n\ge0\]
as desired.

\medskip

\noindent\textbf{Case 2: }$0\le|S|+|T|\le2$.

If $S=T=\emptyset$, then we have $Q_{r}(S,T)=R_{r}(S,T)=0$ (since
$r$ is even). So suppose that $|S|+|T|>0$. Then it follows from (\ref{eq:qbound})
that\[
Q_{r}(S,T)<\frac{n}{\delta-1}\le\frac{3n}{n}=3.\]
If $T\ne\emptyset$, we have\[
R_{r}(S,T)\stackrel{(\ref{eq:rdefn})}{\ge}\delta|T|-1-r|T|\stackrel{(\ref{eq:d-r+ve})}{\ge}3.\]
If $T=\emptyset$, we have $|S|\ge1$ and so by (\ref{eq:rdefn})
we have $R_{r}(S,T)\ge r\ge3$. We therefore have $Q_{r}(S,T)\le R_{r}(S,T)$
in all cases.

\medskip

\noindent\textbf{Case 3:} $|T|\ge r$ or $|S|\ge\delta-r$, but not
both. 

We have\begin{eqnarray}
R_{r}(S,T) & \stackrel{(\ref{eq:rdefn})}{\ge} & (\delta-r)|T|-|S||T|+r|S|\nonumber \\
 & = & (|T|-r)(\delta-r-|S|)+r(\delta-r)\label{eq:**}\\
 & \ge & r(\delta-r)\stackrel{(\ref{eq:d-r+ve})}{\ge}\frac{\epsilon}{4}n^{2}.\nonumber \end{eqnarray}
(Note that (\ref{eq:**}) holds regardless of the values of $|S|$
and $|T|$.) Moreover, we have $Q_{r}(S,T)\le n$. Hence $Q_{r}(S,T)\le R_{r}(S,T)$
as desired.

\medskip

\noindent\textbf{Case 4:} $|T|\ge r$, $|S|\ge\delta-r$, and
$|T|\ne(n+2r-\delta)/2\pm3\sqrt{\epsilon}n$.

The right hand side of (\ref{eq:**}) is clearly minimised when
$|S|+|T|=n$. It therefore suffices to consider this case alone, yielding\begin{align*}
R_{r}(S,T)-Q_{r}(S,T) & \ge(\delta-r)|T|-(n-|T|)|T|+r(n-|T|)-n\\
 & =|T|^{2}+(\delta-2r-n)|T|+n(r-1).\end{align*}
Define a polynomial $f:\mathbb{R}\rightarrow\mathbb{R}$ by\[
f(x)=x^{2}+(\delta-2r-n)x+n(r-1).\]
Suppose this quadratic has real zeroes at $\tau_{1}$ and $\tau_{2}$,
with $\tau_{1}<\tau_{2}$. Then for $|T|\le\tau_{1}$ and $|T|\ge\tau_{2}$,
we must have $R_{r}(S,T)-Q_{r}(S,T)\ge0$. The discriminant $D$ of%
    \COMMENT{$f(x)=(x+\frac{\delta-2r-n}{2})^2-(\delta-2r-n)^2/4+n(r-1)$ and so
the roots of $f$ are $\frac{n+2r-\delta}{2}\pm\frac{\sqrt{(n+2r-\delta)^2-4n(r-1)}}{2}$}
$f$ is given by\begin{eqnarray*}
D = \left(n+2r-\delta\right)^{2}-4n(r-1)=\left(n+2r-\delta\right)^{2}-
\left(1+2(\alpha+\epsilon)+2\sqrt{2(\alpha+\epsilon)}-\frac{4}{n}\right)n^{2}.\end{eqnarray*}
 But\begin{equation}
n+2r-\delta=\left(1+\epsilon+\sqrt{2(\alpha+\epsilon)}\right)n,\label{eq:1}\end{equation}
 so\[
(n+2r-\delta)^{2}=\left(1+\epsilon^{2}+2(\alpha+\epsilon)+2\epsilon+2(1+\epsilon)\sqrt{2(\alpha+\epsilon)}\right)n^{2}\]
 and\[
D=\epsilon\left(\epsilon+2+2\sqrt{2(\alpha+\epsilon)}\right)n^{2}-4n.\]
\begin{eqnarray*}
\end{eqnarray*}
 Hence $0<D\le5\epsilon n^{2}$. In particular, the quadratic does
indeed have two real zeroes $\tau_{1}<\tau_{2}$, and from the quadratic
formula we have\[
\tau_{1}\ge\frac{n+2r-\delta-3\sqrt{\epsilon}n}{2},\qquad\tau_{2}\le\frac{n+2r-\delta+3\sqrt{\epsilon}n}{2}.\]
Since we are in Case~4, we therefore have either $|T|\le\tau_{1}$ or $|T|\ge\tau_{2}$,
and the result follows.

\medskip

\noindent\textbf{Case 5:} $|T|=(n+2r-\delta)/2\pm3\sqrt{\epsilon}n$
and $\delta-r\le|S|\le (n-2r+\delta)/2-3\sqrt{\epsilon}n$.

(Note that our condition on~$|T|$ implies that we cannot have $|S|>(n-2r+\delta)/2+3\sqrt{\epsilon}n$.)
Let $x_{0}:=(n+2r-\delta)/2+3\sqrt{\epsilon}n\ge|T|$. We then have\[
R_{r}(S,T)\stackrel{(\ref{eq:**})}{\ge}(x_{0}-r)(\delta-r-|S|)+r(\delta-r).\]
 Since $x_{0}+|S|\le n$, we may now argue exactly as in Case 4 (with
$x_{0}$ in place of $|T|$) to show that $R_{r}(S,T)\ge Q_{r}(S,T)$.

\medskip

\noindent\textbf{Case 6:} $|T|=(n+2r-\delta)/2\pm 3\sqrt{\epsilon}n$ and $|S|=(n-2r+\delta)/2\pm 3\sqrt{\epsilon}n$.

In this case, we will use the fact that $G$ is not $\eta$-extremal.
From \eqref{eq:1}, we have\[
\left|\frac{n+2r-\delta}{2}-\left(\frac{1}{2}+\sqrt{\frac{\alpha_{+}}{2}}\right)n\right|\le\left(\frac{\epsilon}{2}+\sqrt{\frac{\epsilon}{2}}\right)n.\]
 Since $\epsilon\ll\eta$, we may conclude that\[
\left||T|-\left(\frac{1}{2}+\sqrt{\frac{\alpha_{+}}{2}}\right)n\right|<\eta n.\]
 A similar argument shows that \[
\left||S|-\left(\frac{1}{2}-\sqrt{\frac{\alpha_{+}}{2}}\right)n\right|<\eta n.\]
 Since $G$ is not $\eta$-extremal, this implies that either $e(S,T)\le(1-\eta)|S||T|$
or \[
e(T)\ge\frac{1}{2}\left(\alpha_{+}+\sqrt{\frac{\alpha_{+}}{2}}+\eta\right)n|T|.\]

\medskip

\noindent\textbf{Case 6a:} $e(S,T)\le(1-\eta)|S||T|$.

Then we have\begin{eqnarray*}
R_{r}(S,T)-Q_{r}(S,T) & \stackrel{(\ref{eq:rdefn})}{\ge} & (\delta-r)|T|-(1-\eta)|S||T|+r|S|-n\\
 & \ge & (\delta-r)|T|-(1-\eta)(n-|T|)|T|+r(n-|T|-6\sqrt{\epsilon}n)-n\\
 & = & (1-\eta)|T|^{2}+(\delta-2r-(1-\eta)n)|T|+(1-6\sqrt{\epsilon})nr-n.\end{eqnarray*}
 Write this quadratic as $a|T|^{2}+b|T|+c$, and let the discriminant
be $D$. We then have\begin{align*}
b^{2} & =\left((1-\eta)n+2r-\delta\right)^{2}\stackrel{(\ref{eq:1})}{=}\left(1-\eta+\epsilon+\sqrt{2(\alpha+\epsilon)}\right)^{2}n^{2}\\
 & =\left((1-\eta)^{2}+\epsilon^{2}+2(\alpha+\epsilon)+2(1-\eta)\epsilon+2(1-\eta)\sqrt{2(\alpha+\epsilon)}+2\epsilon\sqrt{2(\alpha+\epsilon)}\right)n^{2}\\
 & \le\left((1-\eta)^{2}+2\alpha+2(1-\eta)\sqrt{2(\alpha+\epsilon)}+\epsilon^{\frac{1}{3}}\right)n^{2}\end{align*}
 and\begin{align*}
4ac & =4(1-\eta)(1-6\sqrt{\epsilon})nr-4(1-\eta)n\\
 & \ge(1-\eta)(1-6\sqrt{\epsilon})\left(1+2(\alpha+\epsilon)+2\sqrt{2(\alpha+\epsilon)}\right)n^{2}-4n\\
 & \ge\left(1-\eta\right)\left(1+2\alpha+2\sqrt{2(\alpha+\epsilon)}\right)n^{2}-\epsilon^{\frac{1}{3}}n^{2}.\end{align*}
 Thus\begin{align*}
D & =b^{2}-4ac\le\left(\left(1-\eta\right)^{2}-\left(1-\eta\right)+2\eta\alpha+2\epsilon^{\frac{1}{3}}\right)n^{2}\\
 & =\left(-\eta\left(1-\eta-2\alpha\right)+2\epsilon^{\frac{1}{3}}\right)n^{2}<0,\end{align*}
 where the last line follows since $\epsilon\ll\eta\ll1/2-\alpha$
and $\alpha<1/2$. Hence this quadratic has no real zeroes, and $R_{r}(S,T)-Q_{r}(S,T)\ge0$
as desired.

\medskip

\noindent\textbf{Case 6b:} $e(T)\ge(\alpha_{+}+\sqrt{\alpha_{+}/2}+\eta)n|T|/2$
and $e(S,T)\ge(1-\eta)|S||T|$. 

Then we have\begin{eqnarray*}
\sum_{v\in T}d(v) & \ge & e(S,T)+2e(T)\\
 & \ge & \left((1-\eta)|S|+\left(\alpha_{+}+\sqrt{\frac{\alpha_{+}}{2}}+\eta\right)n\right)|T|\\
 & \ge & \left((1-\eta)\left(\frac{n-2r+\delta}{2}-3\sqrt{\epsilon}n\right)+\left(\alpha+\sqrt{\frac{\alpha_{+}}{2}}+\eta\right)n\right)|T|\\
 & \stackrel{(\ref{eq:1})}{\ge} & \left((1-\eta)\left(\frac{1}{2}-\frac{\epsilon}{2}-\sqrt{\frac{\alpha+\epsilon}{2}}-3\sqrt{\epsilon}\right)+\alpha+\sqrt{\frac{\alpha_{+}}{2}}+\eta\right)n|T|\\
 & \ge & \left(\frac{1}{2}-\sqrt{\frac{\alpha+\epsilon}{2}}-4\sqrt{\epsilon}-\frac{\eta}{2}+\alpha+\sqrt{\frac{\alpha_{+}}{2}}+\eta\right)n|T|\\
 & \ge & \left(\frac{1}{2}+\frac{\eta}{3}+\alpha\right)n|T|.\end{eqnarray*}
 Hence\begin{eqnarray*}
R_{r}(S,T)-Q_{r}(S,T) 
 & \stackrel{(\ref{eq:rdefn})}{\ge} & \sum_{v\in T}d(v)-(n-|T|)|T|+r(|S|-|T|)-n\\
 & \ge & \sum_{v\in T}d(v)+|T|^{2}-n|T|+r(n-|T|-6\sqrt{\epsilon}n)-r|T|-n\\
 & = & \sum_{v\in T}d(v)+|T|^{2}-(n+2r)|T|+(1-6\sqrt{\epsilon})nr-n\\
 & \ge & |T|^{2}-\left(\left(\frac{1}{2}-\frac{\eta}{3}-\alpha\right)n+2r\right)|T|+(1-6\sqrt{\epsilon})nr-n\\
 & \ge & |T|^{2}-\left(1+\epsilon+\sqrt{2(\alpha+\epsilon)}-\frac{\eta}{3}\right)n|T|+(1-6\sqrt{\epsilon})nr-n\\
 & \ge & |T|^{2}-\left(1+\sqrt{2(\alpha+\epsilon)}-\frac{\eta}{4}\right)n|T|+(1-7\sqrt{\epsilon})nr.\end{eqnarray*}
 Write this quadratic as $|T|^{2}+b|T|+c$, and let the discriminant
be $D$. We then have\begin{align*}
b^{2} & \le\left(1+2\alpha+2\eps +\frac{\eta^{2}}{16}+2\sqrt{2(\alpha+\epsilon)}-\frac{\eta}{2}\right)n^{2}
\le\left(1+2\alpha+2\sqrt{2(\alpha+\epsilon)}-\frac{\eta}{3}\right)n^{2}\end{align*}
 and\begin{align*}
4c & =4(1-7\sqrt{\epsilon})nr=(1-7\sqrt{\epsilon})\left(1+2(\alpha+\epsilon)+2\sqrt{2(\alpha+\epsilon)}\right)n^{2}\\
 & \ge\left(1+2\alpha+2\sqrt{2(\alpha+\epsilon)}\right)n^{2}-\epsilon^{\frac{1}{3}}n^{2}.\end{align*}
 Thus\[
D=b^{2}-4c\le \left(\epsilon^{\frac{1}{3}}-\frac{\eta}{3}\right)n^{2}<0\]
 since $\epsilon\ll\eta$. Hence this quadratic has no real zeroes,
and $R_{r}(S,T)-Q_{r}(S,T)\ge0$ as desired. This completes the proof.
\end{proof}
It is now simple to prove that every non-extremal graph $G$ whose minimum degree $\delta$ is slightly larger than $n/2$ contains
significantly more than $\textnormal{reg}_{\textnormal{even}}(n,\delta)/2$ edge-disjoint Hamilton cycles.
\begin{lem}
\label{lem:mainresultpart2}Suppose $0<1/n\ll c\ll\eta\ll\alpha,1/2-\alpha<1/2$.
Let $G$ be a graph on $n$ vertices with $\delta:=\delta(G)=\left(1/2+\alpha\right)n$
such that $G$ is not $\eta$-extremal. Then $G$ contains at least $\textnormal{reg}_{\textnormal{even}}(n,\delta)/2+cn$
edge-disjoint Hamilton cycles.\end{lem}
\begin{proof}
Let $\tau_{0}:=\tau(1/4)$ be as defined in Theorem~\ref{thm:hamdecresult}.
Choose new constants $\epsilon,\eps',\nu,\nu',\tau$ such that\[
0<1/n\ll\nu',c\ll\epsilon,\epsilon'\ll\eta\ll\nu\ll\tau\ll\alpha,1/2-\alpha,\tau_{0}.\]
Let\[
r:=\left(\frac{1}{4}+\frac{\alpha+\epsilon'}{2}+\sqrt{\frac{\alpha+\epsilon'}{2}}\right)n.\]
By reducing $\epsilon'$ and $\eps$ slightly if necessary we may assume that both $r$ and $\eps n$
are even integers. By Lemmas \ref{lem:anclem1} and \ref{lem:smallexpanderfactor},
$G$ contains an $\epsilon n$-factor $H$ which is a robust $(\nu',\tau)$-expander.
Let $G':=G-H$. By Lemma~\ref{lem:removefactornotextremal}, $G'$
is not $(\eta/2)$-extremal. Since also $\delta(G')=(1/2+\alpha-\eps)n$, we can apply Lemma~\ref{lem:largefactor}
with $\eps+\eps'$ and $\alpha-\eps$ playing the roles of $\eps $ and $\alpha$ to find an $r$-factor $H'$ of~$G'$.

Since $H$ is a robust $(\nu',\tau)$-expander (and thus also a robust
$(\nu',\tau_{0})$-expander), the same holds for $H+H'$. Hence by
Theorem~\ref{thm:hamdecresult}, $H+H'$ can be decomposed into $d(H+H')/2$
edge-disjoint Hamilton cycles. By Theorem~\ref{thm:regbounds} we have
$r\ge\textnormal{reg}_{\textnormal{even}}(n,\delta)$,
and so\[
\frac{1}{2}d(H+H')\ge\frac{1}{2}\left(\textnormal{reg}_{\textnormal{even}}(n,\delta)+\epsilon n\right)\ge
\frac{1}{2}\textnormal{reg}_{\textnormal{even}}(n,\delta)+cn\]
as desired.
\end{proof}

\section{Proof of Theorems~\ref{thm:mainresult} and~\ref{thm:halfnresult}}\label{sec:nhalf}

We first combine Lemmas~\ref{lem:mainresultpart1} and~\ref{lem:mainresultpart2} to prove Theorem~\ref{thm:mainresult}.

\removelastskip\penalty55\medskip\noindent{\bf Proof of Theorem~\ref{thm:mainresult}. }
Choose $n_0\in\mathbb{N}$ and an additional constant $\eta$ such that $1/n_0\ll \eta\ll \eps$. Define $\alpha$ by $\delta(G)=(1/2+\alpha)n$.
Recall from Section~\ref{sec:intro} that Theorem~\ref{thm:mainresult} was already proved in~\cite{Kellyapps}
for the case when $\delta(G)\ge (2-\sqrt{2}+\eps)n$. So we may assume that $\alpha\le 3/2-\sqrt{2}+\eps$
and so $\eta\ll \alpha,1/2-\alpha$. Thus we can apply Lemma~\ref{lem:mainresultpart1} (if $G$ is $\eta$-extremal) or
Lemma~\ref{lem:mainresultpart2} (if $G$ is not $\eta$-extremal)
to find ${\rm reg}_{\rm even}(n,\delta(G))/2$ edge-disjoint Hamilton cycles in $G$.
\endproof

Let $G$ be a graph on $n$ vertices whose minimum degree is not much smaller than $n/2$. 
Before we can prove Theorem~\ref{thm:halfnresult}, we must first show that either $G$ is a robust expander or it is close 
to either the complete bipartite graph $K_{n/2,n/2}$ or the disjoint union $K_{n/2}\dot\cup K_{n/2}$ of two cliques. 
The former case corresponds to (i) of Theorem~\ref{thm:halfnresult}, and the latter case corresponds to (ii).

\begin{defn}
\label{def:closedefn}We say that a graph $G$ is \emph{$\epsilon$-close
to $K_{n/2,n/2}$} if there exists $A\subset V(G)$ with $|A|=\lfloor n/2\rfloor$ and such that
$e(A)\le \epsilon n^{2}$.
We say that $G$ is \emph{$\epsilon$-close to $K_{n/2}\dot{\cup}K_{n/2}$}
if there exists $A\subset V(G)$ with $|A|=\lfloor n/2\rfloor$ and such that $e(A,\overline{A})\le\epsilon n^{2}$.
\end{defn}
Suppose that $G$ is a graph of minimum degree roughly $n/2$. If $G$ is $\epsilon$-close to $K_{n/2,n/2}$ then the
bipartite subgraph of $G$ induced by $A$ and $\overline{A}$ is almost complete. However, $\overline{A}$
may also contain many edges. If $G$ is $\epsilon$-close to $K_{n/2}\dot{\cup}K_{n/2}$ then both $G[A]$ and $G[\overline{A}]$
are almost complete.
\begin{lem}
\label{lem:characteriseexpanders}Suppose $0<1/n\ll\kappa\ll\nu\ll\tau,\epsilon<1$.
Let $G$ be a graph on $n$ vertices of minimum degree $\delta:=\delta(G)\ge(1/2-\kappa)n$.
Then $G$ satisfies one of the following properties:
\begin{enumerate}
\item $G$ is $\epsilon$-close to $K_{n/2,n/2}$;
\item $G$ is $\epsilon$-close to $K_{n/2}\dot{\cup}K_{n/2}$;
\item $G$ is a robust $(\nu,\tau)$-expander.%
    \COMMENT{Quick informal reminder of how this works for my own benefit: we can
assume $|S|\approx n/2$ or it's easy. If any non-tiny part of $S$
isn't in $RN$, it'll spam out edges to $\overline{S}$ and so most
of $\overline{S}$ will be in $RN$. This will give us $|RN|\ge|S|+\nu n$
unless almost no vertices in $S$ are in $RN$, in which case $S$
must be really sparse. This in turn implies that $e(S,\overline{S})$
must be large by $\delta\ge n/2$, and we're in the bipartite case.
Alternatively, if almost all of $S$ is in $RN$, then we have $|RN|\ge|S|+\nu n$
unless almost none of $\overline{S}$ is in $RN$. This can only happen
if $e(S,\overline{S})$ is really small, which in turn implies that
$S$ and $\overline{S}$ are almost cliques since $\delta\approx|S|,|\overline{S}|$.
Finally, we add and remove on the order of $\nu n$ arbitrary vertices
to $S$ to give it size exactly $\lfloor n/2\rfloor$.}
\end{enumerate}
\end{lem}
\begin{proof}
Suppose $S\subset V(G)$ with $\tau n\le|S|\le(1-\tau)n$. Our aim
is to show that either $RN:=RN_{\nu}(S)$ has size at least $|S|+\nu n$
or that $G$ is close to either $K_{n/2,n/2}$ or $K_{n/2}\dot{\cup}K_{n/2}$.
We will split the proof into cases depending on $|S|$.

\medskip

\noindent\textbf{Case 1:} $\tau n\le|S|\le(1/2-\sqrt{\nu})n$.

In this case, we have\[
\delta|S|\le e'(S,V(G))=e'(S,RN)+e'(S,\overline{RN})\le|S||RN|+\nu n^{2}\le|S||RN|+\nu n\frac{|S|}{\tau},\]
and so $|RN|\ge(1/2-\kappa-\nu/\tau)n\ge|S|+\nu n$ as desired. (Recall that $e'(A,B)$ denotes
the number of ordered pairs $(a,b)$ with $ab\in E(G)$, $a\in A$ and $b\in B$.)%
    \COMMENT{We have this in the notation section, but since it's not standard, it seems to be better to recall it.}

\medskip

\noindent\textbf{Case 2:} $(1/2+2\nu)n\le|S|\le(1-\tau)n$.

In this case, we have $RN=V(G)$ and so the result is immediate. Indeed,
for all $v\in V(G)$, we have $d(v)\ge(1/2-\kappa)n$ and so $|N(v)\cap S|\ge(2\nu-\kappa)n\ge\nu n$.

\medskip

\noindent\textbf{Case 3:} $(1/2-\sqrt{\nu})n\le|S|\le(1/2+2\nu)n$.

Suppose that $|RN|<|S|+\nu n$. We will first show that either $|S\setminus RN|<\sqrt{\nu}n$
or $G$ is $\epsilon$-close to $K_{n/2,n/2}$. Suppose $|S\setminus RN|\ge\sqrt{\nu}n$.
Then\begin{align*}
|S\setminus RN|(\delta-\nu n) & \le e(S\setminus RN,\overline{S})=e(S\setminus RN,\overline{S}\cap RN)+e(S\setminus RN,\overline{S}\setminus RN)\\
 & \le|S\setminus RN||\overline{S}\cap RN|+\nu n^{2}
\le|S\setminus RN||\overline{S}\cap RN|+\sqrt{\nu}n|S\setminus RN|,\end{align*}
and so $|\overline{S}\cap RN|\ge\delta-2\sqrt{\nu}n.$ But then together
with our assumption that $|RN|<|S|+\nu n$, this implies $|S\cap RN|<3\sqrt{\nu}n$.
Hence $e(S)\le3\sqrt{\nu}n^{2}+|S|\nu n<4\sqrt{\nu}n^{2}$. By adding
or removing at most $\sqrt{\nu}n$ arbitrary vertices to or from $S$,
we can form a set $A$ of $\lfloor n/2\rfloor$ vertices with \[
e(A)<4\sqrt{\nu}n^{2}+\sqrt{\nu}n^{2}=5\sqrt{\nu}n^{2}\le\epsilon n^{2}.\]
Thus $G$ is $\epsilon$-close to $K_{n/2,n/2}$.

We may therefore assume that $|S\setminus RN|<\sqrt{\nu}n$, from
which it follows that $|\overline{S}\cap RN|<2\sqrt{\nu}n$ (by our
initial assumption that $|RN|<|S|+\nu n$). We will now show that
$G$ is $\epsilon$-close to $K_{n/2}\dot{\cup}K_{n/2}$. We have
$e(S,\overline{S}\cap RN)\le|S||\overline{S}\cap RN|\le2\sqrt{\nu}n^{2}$,
and hence $e(S,\overline{S})\le3\sqrt{\nu}n^{2}$. 
As before, by adding or removing at most $\sqrt{\nu}n$ arbitrary
vertices to or from $S$, we can therefore form a set $A$ of $\lfloor n/2\rfloor$
vertices with $e(A,\overline{A})\le e(S,\overline{S})+\sqrt{\nu}n^{2}\le\epsilon n^{2}$.
Hence $G$ is $\epsilon$-close to $K_{n/2}\dot{\cup}K_{n/2}$.

If $G$ is not $\epsilon$-close to either $K_{n/2,n/2}$ or $K_{n/2}\dot{\cup}K_{n/2}$,
we must therefore have $|RN|\ge|S|+\nu n$ for all $S\subset V(G)$
with $\tau n\le|S|\le(1-\tau)n$, so that $G$ is a robust $(\nu,\tau)$-expander
as required.
\end{proof}
We now have all the tools we need to prove Theorem~\ref{thm:halfnresult}.

\removelastskip\penalty55\medskip\noindent{\bf Proof of Theorem~\ref{thm:halfnresult}. }
Let $\tau:=\tau(1/4)$ be as defined in Theorem~\ref{thm:hamdecresult}.
Choose $n_0\in\mathbb{N}$ and new constants $\epsilon',\epsilon'',\nu,\nu'$
such that\[
0<1/n_{0}\ll\nu'\ll\eps\ll\eps',\eps''\ll\nu\ll\tau,\eta.\]
Consider any graph $G$ on $n\ge n_0$ vertices as in Theorem~\ref{thm:halfnresult}. Let $\delta:=\delta(G)$
and define $\alpha$ by $\delta=(1/2+\alpha)n$. So $-\eps\le \alpha\le \eps$.
Let\[
r:=\left(\frac{1}{4}+\frac{\alpha+\epsilon'}{2}+\sqrt{\frac{\alpha+\epsilon'}{2}}\right)n.\]
By reducing $\epsilon'$ and $\eps''$ slightly if necessary we may assume that both $r$ and $\eps''n$ are
even integers.
 
Suppose that $G$ does not satisfy~(i), i.e. $e(X)\ge\eta n^{2}$
for all $X\subset V(G)$ with $|X|=\lfloor n/2\rfloor$. We claim
that $G$ is not $(\eta/4)$-extremal.
To show this, consider any set $B\subset V(G)$ with
$$
|B|=\left(\frac{1}{2}+\sqrt{\frac{\alpha_+}{2}}\pm\frac{\eta}{4}\right)n.
$$
By adding or removing at most $\eta n/2$ arbitrary vertices to and
from $B$, we obtain a set $B'$ with $|B'|=\lfloor n/2\rfloor$
and such that $e(B') \le e(B)+\eta n^{2}/2$.
Together with our assumption that (i) does not hold, this implies
that
$$e(B)\ge\frac{\eta n^{2}}{2}\ge \frac{1}{2}\left(\alpha_+ +\sqrt{\frac{\alpha_+}{2}}+\frac{\eta}{4} \right)n|B|.
$$
Hence $G$ is not $(\eta/4)$-extremal.

Suppose moreover that (ii) does not hold, so that $G$ fails to be
$\eta$-close to $K_{n/2}\dot{\cup}K_{n/2}$. By Lemma~\ref{lem:characteriseexpanders},
it follows that $G$ is a robust $(\nu,\tau)$-expander. By Lemma~\ref{lem:smallexpanderfactor},
$G$ therefore contains an $\epsilon'' n$-factor $H$ which is a robust $(\nu',\tau)$-expander.
Let $G':=G-H$. By Lemma~\ref{lem:removefactornotextremal}, $G'$
is not $(\eta/8)$-extremal. Since also $\delta(G')=(1/2+\alpha-\eps'')n$, we can apply Lemma~\ref{lem:largefactor}
with $\eps'+\eps''$ and $\alpha-\eps''$ playing the roles%
   \COMMENT{ok since $-(\eps'+\eps'')\le \alpha-\eps''$ as $-\eps'\le -\eps\le \alpha$}
of $\eps$ and $\alpha$ to find an $r$-factor $H'$ of~$G'$.

Since $H$ is a robust $(\nu',\tau)$-expander, the same holds for $H+H'$.
Hence by Theorem~\ref{thm:hamdecresult}, $H+H'$ can be decomposed into $d(H+H')/2$
edge-disjoint Hamilton cycles. By Theorem~\ref{thm:regbounds} (and the fact that
$\textnormal{reg}_{\textnormal{even}}(n,\delta)=0$
if $\delta<n/2$) we have
$r\ge\max\{\textnormal{reg}_{\textnormal{even}}(n,\delta) ,n/8\}$,
and so\begin{align*}
\frac{1}{2}d(H+H') & \ge\frac{1}{2}\left( \max \{\textnormal{reg}_{\textnormal{even}}(n,\delta),n/8\}+\epsilon'' n\right)
\ge\frac{1}{2}\max \{\textnormal{reg}_{\textnormal{even}}(n,\delta),n/8\}+\eps n,
\end{align*}
as desired.\endproof

\medskip

{\footnotesize \obeylines \parindent=0pt

Daniela K\"{u}hn, John Lapinskas, Deryk Osthus 
School of Mathematics 
University of Birmingham
Edgbaston
Birmingham
B15 2TT
UK
\begin{flushleft}
{\it{E-mail addresses}:\\
\tt{\{d.kuhn, jal129, d.osthus\}@bham.ac.uk}}
\end{flushleft}}

\end{document}